\documentclass{amsart}
\usepackage[russian,english]{babel}
\numberwithin{equation}{section}

\newtheorem{thm}{Theorem}[section]
\newtheorem{lem}[thm]{Lemma}
\newtheorem{cor}[thm]{Corollary}

\newtheorem{defin}[thm]{Definition}

\begin{document}
\title{Fractional Telegraph
equation with the Caputo derivative}

\author{Ravshan Ashurov}
\author{Rajapboy Saparbayev}
\address{Institute of Mathematics, Uzbekistan Academy of Science,
Tashkent, Student Town str. 100174} \email{ashurovr@gmail.com}
\curraddr{Institute of Mathematics, Uzbekistan Academy of Science,
Tashkent, Student Town str. 100174} \email{rajapboy1202@gmail.com}

\small

\title[ Fractional Telegraph
equation with the Caputo derivative] { Fractional Telegraph
equation with the Caputo derivative }

\begin{abstract}
 The Cauchy problem for the telegraph equation
$(D_{t}^{\rho })^{2}u(t)+2\alpha D_{t}^{\rho }u(t)+Au(t)=f(t)$  ($0<t\leq T, \, 0<\rho<1$), with the Caputo derivative is considered.
Here $A$ is a selfadjoint positive operator,  acting in a Hilbert space $H$, $D_t$ is the Caputo fractional derivative. Existence and uniqueness theorems for
the solution to the problem under consideration is proved. Inequalities of stability are obtained.

\vskip 0.3cm \noindent {\it AMS 2000 Mathematics Subject
Classifications} :
Primary 35R11; Secondary 34A12.\\
{\it Key words}: Telegraph type equations, the  Caputo derivatives, time-dependent source identification problem.
\end{abstract}

\maketitle

\section{\textbf{Introduction}}

Let $H$ be a separable  Hilbert space with the inner product $(\cdot, \cdot)$ and the norm $||\cdot||$. Let $A: H\rightarrow H$
be an arbitrary unbounded positive selfadjoint operator acting in $H$ with the domain of definition $D(A)$.
Suppose that $A$ has a complete orthonormal system of eigenfunctions
$\{v_k\}$ and a countable set of positives
eigenvalues $\lambda_k$. It is convenient to assume that
the eigenvalues do not decrease as their number increases, i.e.
$0<\lambda_1\leq\lambda_2 \cdot\cdot\cdot\rightarrow +\infty$.

For vector functions (or just functions)
$h: \mathbb{R}_+\rightarrow H$ fractional integrals and derivatives
are defined similarly with scalar functions and known formulas and
properties are preserved \cite{Liz}. Recall that fractional integrals of order $\sigma<0$ of the
function $h(t)$ defined on $[0,\infty)$ has the form (see, for example, \cite{Pskhu})
\begin{equation}\label{def0}
J_t^\sigma h(t)=\frac{1}{\Gamma
(-\sigma)}\int\limits_0^t\frac{h(\xi)}{(t-\xi)^{\sigma+1}} d\xi,
\quad t>0,
\end{equation}
provided the right-hand side exists. Here $\Gamma(\sigma)$ is
Euler's gamma function. Using this definition one can define the
Caputo fractional derivative of order $\rho\in (0,1)$:
$$
D_t^\rho h(t)= J_t^{\rho-1}\frac{d}{dt} h(t).
$$

Note that if $\rho=1$, then the fractional derivative coincides with
the ordinary classical derivative of the first order: $D_t h(t)= \frac{d}{dt} h(t)$.

Let $C[0, T]$ be the set of continuous functions defined on $[0,T]$ with the standard max-norm $||\cdot||_{C[0,T]}$ and let $C(H)=C([0,T]; H)$ stand
for a space of continuous $H$-valued functions $h(t)$ defined on $[0,T]$, and equipped with the norm
\[
||h||_{C(H)}=\max\limits_{0\leq t\leq T}||h(t)||.
\]

Let $\rho \in (0,1)$ be a fixed number. Consider the following Cauchy problem
\begin{equation}\label{prob1}
\begin{cases}
  & (D_{t}^{\rho })^{2}u(t)+2\alpha D_{t}^{\rho }u(t)+Au(t)=f(t),  0<t\le T; \\
 & \underset{t\to 0}{\mathop{\lim }}\,D_{t}^{\rho }u(t)={{\varphi }_{0}},\\
 & u(0)={{\varphi }_{1}}, \\
\end{cases}
\end{equation}
where  $f(t)\in C(H)$ and $\varphi_{0},\varphi_{1}  $ are known elements of $H$.

\begin{defin}\label{def1} If  function $u(t)$ with the properties
$(D_{t}^{\rho })^{2}u(t), Au(t)\in C((0,T]; H)$ and  $u(t),D^{\rho}_{t}u(t)\in C(H)$  and satisfying conditions
(\ref{prob1}) is called \textbf{the
 solution} of the forward problem.
\end{defin}

In order to formulate the main results of this paper, for an arbitrary real number $ \tau $  we introduce the power
of operator $ A $, acting in $ H $ according to the rule
$$
A^\tau h= \sum\limits_{k=1}^\infty \lambda_k^\tau h_k v_k.
$$
The domain of definition of this operator has the
form
$$
D(A^\tau)=\{h\in H:  \sum\limits_{k=1}^\infty \lambda_k^{2\tau}
|h_k|^2 < \infty\}.
$$
It's immediate from this definition that $D(A^\tau)\subset D(A^\sigma)$ for any $\tau\ge \sigma$.

For elements of $D(A^\tau)$ we define the norm
\[
||h||^2_\tau=\sum\limits_{k=1}^\infty \lambda_k^{2\tau} |h_k|^2 =
||A^\tau h||^2,
\]
and together with this norm $D(A^\tau)$ turns into a Hilbert
space.
\begin{thm}\label{main1} Let $\alpha>0$, $\varphi_{0}\in H$ and $\varphi_{1}\in D(A^{\frac{1}{2}})$.Further, let $\epsilon\in (0,1)$ be any fixed number $f(t)\in C([0,T]; D(A^{\epsilon}))$. Then the forward problem has a unique solution.

Moreover, there is a constant $C>0$ such that the following stability estimate
\[
   ||(D_{t}^{\rho})^{2} u|| +||D_t^\rho u|| + ||A u||\leq C \bigg[t^{-\rho}\big(||\varphi_{0}||+||\varphi_{1}||_\frac{1}{2}\big)
+ \max\limits_{0\leq t\leq T}||f(t)||_{\epsilon}\bigg], \quad t>0,
\]
holds.

\end{thm}
The telegraph equation first appeared in the work of Oliver Heaviside in 1876. When simulating the passage of electrical signals in marine telegraph cables, he obtained the equation
\[
u_{tt}+a u_t+b u-cu_{xx}=0,
\]
where $a, b$ are nonnegative constants and $c$ is a positive constant (see, e.g.\cite{Lieberstein}, \cite{W.Arendt}). Then specialists came to this equation when modeling various physical processes. A small overview of various applications of the telegraph equation is given in \cite{Jordan}.
For example, it appears in the theory of superconducting electrodynamics, where it describes the propagation
electromagnetic waves in superconducting media (see, e.g.\cite{Wallace}). In \cite{Jordan}, the propagation of digital and analog signals through media which, in general, are
both dissipative and dispersive is modeled using the telegraph equation. Some applications of the telegraph equation to the theory of random walks are contained in \cite{Banasiak}. Another field of application of the telegraph equation is the biological sciences (see, e.g. \cite{Lieberstein}, \cite{Debnath}, \cite{Goldstein}).

In recent decades, fractional calculus has attracted the attention of many mathematicians and researchers as non-integer derivative operators have come to play a larger role in describing physical phenomena, modeling more accurately and efficiently than classical derivatives \cite{SU, AOLob, AODif}. Various forms of the time-fractional telegraph equation was considered by a number of researchers (see e.g.\cite{Doetsch}, \cite{Hosseini}, \cite{Ashyr3}, \cite{Hashmi}). Thus, in the fundamental work R.C. Cascaval et al.\cite{RC}, the asymptotic behavior of the solution $u(t)$ to problem (\ref{prob1}) with a homogenuous equation for large $t$ was studied. The authors succeeded in proving the existence of a solution $v(t)$ to equation $2\alpha D_{t}^{\rho }v(t)+Av(t)=0$ for which the asymptotic relation
\[
u(t)=v(t) + o(v(t)), \quad t\to +\infty,
\]
is valid.

In works \cite{Orsingher} (in the case of $\rho=1/2$), \cite{Beghin} (in the case of fractional
derivatives of rational order $\rho=m/n$ with $m < n$), fundamental solutions for problem (\ref{prob1}) are constructed. In these papers, the elliptic part of the equation has the form $Au(x,t)= u_{xx}(x,t)$.

A number of specialists have developed efficient and optimally accurate numerical algorithms for solving the problem  (\ref{prob1}) for different operators $A$. A review of some
works in this direction is contained in the papers \cite{Jordan},\cite{Momani}.

\section{\textbf {Preliminaries}}

In this section, we  recall some information about Mittag-Leffler functions, differential and integral equations, which we will use in the following sections.

For $0 < \rho < 1$ and an arbitrary complex number $\mu$, by $
E_{\rho, \mu}(z)$ we denote the Mittag-Leffler function  of complex argument $z$ with two
parameters:
\begin{equation}\label{ml}
E_{\rho, \mu}(z)= \sum\limits_{k=0}^\infty \frac{z^k}{\Gamma(\rho
k+\mu)}.
\end{equation}
If the parameter $\mu =1$, then we have the classical
Mittag-Leffler function: $ E_{\rho}(z)= E_{\rho, 1}(z)$.
Prabhakar (see, \cite{Prah}) introduced  the function $E^{\gamma}_{\rho, \mu}(z)$ of the form
\begin{equation}\label{ml1}
E^{\gamma}_{\rho, \mu}(z)= \sum\limits_{k=0}^\infty \frac{(\gamma)_{k}}{\Gamma(\rho k+\mu)}\cdot\frac{z^{k}}{k!},
\end{equation}
where $z \in C$, $\rho$, $\mu$ and $\gamma$ are arbitrary positive constants, and $(\gamma)_{k}$ is the Pochhammer
symbol. When $\gamma = 1$, one has $E^{1}_{\rho,\mu}(z)=E_{\rho,\mu}(z)$. We also have
\begin{equation}\label{Prahma}
E^{2}_{\rho,\mu}(z)=\frac{1}{\rho}\big[E_{\rho,\mu-1}(z)+(1-\rho+\mu)E_{\rho,\mu}(z)\big].
\end{equation}
Since $E_{\rho, \mu}(z)$  is an analytic function of $z$, then it is bounded for $|z|\leq 1$. On the other hand the well known asymptotic estimate of the
Mittag-Leffler function has the form (see, e.g.,
\cite{Dzh66}, p. 133):
\begin{lem}\label{ml} Let  $\mu$  be an arbitrary complex number. Further let $\beta$ be a fixed number, such that $\frac{\pi}{2}\rho<\beta<\pi \rho$, and $\beta \leq |\arg z|\leq \pi$. Then the following asymptotic estimate holds
\[
E_{\rho, \mu}(z)= -\frac{z^{-1}}{\Gamma(\rho-\mu)} + O(|z|^{-2}),
\,\, |z|>1.
\]
\end{lem}
\begin{cor}\label{cor} Under the conditions of Lemma \ref{ml} one has
\[
|E_{\rho,\mu}(z)|\le \frac{M}{1+|z|}, \quad |z|\ge0,
\]
where  $M$-constant, independent of $z$.
\end{cor}
We also use the following estimate for sufficiently large $\lambda>0$  and $\alpha>0$, $0<\epsilon<1$:
\begin{equation}\label{Large}
  |t^{\rho-1}E_{\rho,\mu}(-(\alpha-\sqrt{\alpha^{2}-\lambda})t^{\rho})|\leq\frac{t^{\rho-1}M}{1+\sqrt{\lambda}t^{\rho}}\leq M\lambda^{\epsilon-\frac{1}{2}}t^{2\epsilon\rho-1}, \quad  t>0,
\end{equation}
which is easy to verify. Indeed, let $(\lambda)^{\frac{1}{2}} t^{\rho}<1$, then $t<\lambda^{-\frac{1}{2\rho}}$ and
\[
t^{\rho-1}=t^{\rho-2\epsilon\rho}t^{2\epsilon\rho-1}<\lambda^{\epsilon-\frac{1}{2}}t^{2\epsilon\rho-1}.
\]
If $(\lambda)^{\frac{1}{2}} t^{\rho}\geq 1$, then $\lambda^{-\frac{1}{2}}\leq t^{\rho}$ and
\[
\lambda^{-\frac{1}{2}}t^{-1}=\lambda^{\epsilon-\frac{1}{2}}\lambda^{-\epsilon}t^{-1}\leq \lambda^{\epsilon-\frac{1}{2}}t^{2\rho\epsilon-1}.
\]
\begin{lem} If $\rho>0$ and $\lambda \in C$, then (see \cite{Gorenflo},  p.446)
\begin{equation}\label{Dervitavi}
D^{\rho}_{t}E_{\rho,1}(\lambda t^{\rho})=\lambda E_{\rho,1}(\lambda t^{\rho}) \quad t>0 .
\end{equation}
\end{lem}
The following lemma would be an extension of the result of \cite{RC}. He has shown the case $\alpha^2\neq \lambda$ for  homogeneous equation. We have been able to prove this extension using similar ideas.
\begin{lem}\label{eq2} Let $g(t)\in C[0,T]$ and $\varphi_{0},\varphi_{1} $ are known numbers. Then the unique solution of the Cauchy problem
\begin{equation}\label{prob2}
\begin{cases}
  & (D_{t}^{\rho })^{2}y(t)+2\alpha D_{t}^{\rho }y(t)+\lambda y(t)=g(t), \quad  0<t\le T; \\
 & \underset{t\to 0}{\mathop{\lim }}\,D_{t}^{\rho }y(t)={{\varphi }_{0}},\\
 & y(0)={{\varphi }_{1}}, \\
\end{cases}
\end{equation}
has the form
\begin{equation}\label{1}y(t)=
\begin{cases}
  & y_{1}(t),\quad \alpha^{2}\neq \lambda; \\
 & y_{2}(t),\quad \alpha^{2}=\lambda.
\end{cases}
\end{equation}
Here
$$
y_{1}(t)=\frac{(\sqrt{\alpha^{2}-\lambda}+\alpha)\varphi_{1}}{2\sqrt{\alpha^{2}-\lambda}}E_{\rho,1}((-\alpha+\sqrt{\alpha^{2}-\lambda})t^{\rho})+\frac{(\sqrt{\alpha^{2}-\lambda}-\alpha)\varphi_{1}}{2\sqrt{\alpha^{2}-\lambda}}E_{\rho,1}((-\alpha-\sqrt{\alpha^{2}-\lambda})t^{\rho})
$$
$$
+\frac{1}{2\sqrt{\alpha^{2}-\lambda}}\big[E_{\rho,1}((-\alpha+\sqrt{\alpha^{2}-\lambda})t^{\rho})-E_{\rho,1}((-\alpha-\sqrt{\alpha^{2}-\lambda})t^{\rho})  \big]\varphi_{0}
$$
$$
+\frac{1}{2\sqrt{\alpha^{2}-\lambda}}\int_{0}^{t}(t-\tau)^{\rho-1}E_{\rho,\rho}((-\alpha+\sqrt{\alpha^{2}-\lambda})(t-\tau)^{\rho})g(\tau)d\tau
$$
$$
-\frac{1}{2\sqrt{\alpha^{2}-\lambda}}\int_{0}^{t}(t-\tau)^{\rho-1}E_{\rho,\rho}((-\alpha-\sqrt{\alpha^{2}-\lambda})(t-\tau)^{\rho})g(\tau)d\tau,
$$
$$
y_{2}(t)=t^{\rho}E^{2}_{\rho,1+\rho}(-\alpha t^{\rho})\varphi_{0}+\alpha t^{\rho} E^{2}_{\rho,1+\rho}(-\alpha t^{\rho})\varphi_{1}
$$
$$
+E_{\rho,1}(-\alpha t^{\rho})\varphi_{1}+\int_{0}^{t}(t-\tau)^{2\rho-1}E^{2}_{\rho,2\rho}(-\alpha(t-\tau)^{\rho})g(\tau)d\tau.
$$
\end{lem}
\begin{proof} We use the Laplace transform to prove the lemma. Let us remind that the Laplace transform of a function $f(t)$ is defined as (see \cite{Mainardi})
\[
L[f](s)=\hat{f}(s)=\int_{0}^{\infty}e^{-st}f(t)dt.
\]
The inverse Laplace transform is defined by
\[
L^{-1}[\hat{f}](t)=\frac{1}{2\pi i}\int_{C}e^{st}\hat{f}(s)ds,
\]
where $C$ is a contour parallel to the imaginary axis and to the right of the singularities of  $\hat{f}$.

Let us apply the Laplace transform to equation (\ref{prob2}).  Then equation (\ref{prob2}) becomes:
\[
s^{2\rho}\hat{y}(s)+2\alpha s^{\rho}\hat{y}(s)+\lambda\hat{y}(s)-s^{2\rho-1}y(0)-s^{\rho-1}\lim_{t\to 0}D^{\rho}_{t}y(t)-2\alpha s^{\rho-1}y(0)=\hat{g}(s),
\]
it follows from this
\begin{equation}\label{laplas}
\hat{y}(s)=\frac{\hat{g}(s)+s^{2\rho-1}{y}(0)+s^{\rho-1}\lim_{t\to 0}D^{\rho}_{t}y(t)+2\alpha s^{\rho-1}y(0)}{s^{2\rho}+2\alpha s^{\rho}+\lambda}.
\end{equation}
\textbf{Case 1}. Let  $\alpha^{2}\neq \lambda$.

Write $\hat{y}(s)=\hat{y}_{0}(s)+\hat{y}_{1}(s)$, where
\[
\hat{y}_{0}(s)=\frac{(s^{2\rho-1}+2\alpha s^{\rho-1})\varphi_{0}+s^{\rho-1}\varphi_{1}}{{s^{2\rho}+2\alpha s^{\rho}}+\lambda}, \hat{y}_{1}(s)=\frac{\hat{g}(s)}{{s^{2\rho}+2\alpha s^{\rho}}+\lambda},
\]
furthemore
\[
y(t)=L^{-1}[\hat{y}_{0}(s)]+L^{-1}[\hat{y}_{1}(s)].
\]
As in the work of \cite{RC} when we apply inverse Laplace transform we get the following expression:
\begin{equation}\label{y0}
L^{-1}[\hat{y}_{0}(s)]=\frac{(\sqrt{\alpha^{2}-\lambda}+\alpha)\varphi_{1}+\varphi_{0}}{2\sqrt{\alpha^{2}-\lambda}}E_{\rho,1}\big((-\alpha+\sqrt{\alpha^{2}-\lambda})t^{\rho}\big)
\end{equation}
$$
+\frac{(\sqrt{\alpha^{2}-\lambda}-\alpha)\varphi_{1}-\varphi_{0}}{2\sqrt{\alpha^{2}-\lambda}}E_{\rho,1}\big((-\alpha-\sqrt{\alpha^{2}-\lambda})t^{\rho}\big).
$$
For the second term on the right one can obtain the inverse by splitting the function into simpler functions.
\begin{equation}\label{y1}
L^{-1}[\hat{y}_{1}(s)]=L^{-1}\bigg[\frac{\hat{g}(s)}{s^{2\rho}+2\alpha s^{\rho}+\lambda}\bigg]=L^{-1}\bigg[\frac{1}{s^{2\rho}+2\alpha s^{\rho}+\lambda}\bigg]\ast L^{-1}[\hat{g}(s)].
\end{equation}
By $f \ast g$ we denoted the Laplace convolution of functions defined by  $(f\ast g)(t)=\int_{0}^{t}f(\tau)g(t-\tau)d\tau$.

The initial component can be easily obtained by employing the subsequent straightforward observations.
\begin{equation*}
\begin{split}
    L^{-1}\bigg[&\frac{1}{s^{2\rho}+2\alpha s^{\rho}+\lambda}\bigg] = L^{-1}\bigg[\frac{1}{2\sqrt{\alpha^{2}-\lambda}}\bigg(\frac{1}{s^{\rho}+\alpha-\sqrt{\alpha^{2}-\lambda}}-\frac{1}{s^{\rho}+\alpha+\sqrt{\alpha^{2}-\lambda}}\bigg)\bigg]\\
    &=\frac{1}{2\sqrt{\alpha^{2}-\lambda}}L^{-1}\bigg[\frac{1}{s^{\rho}+\alpha-\sqrt{\alpha^{2}-\lambda}}\bigg]-\frac{1}{2\sqrt{\alpha^{2}-\lambda}}L^{-1}\bigg[\frac{1}{s^{\rho}+\alpha+\sqrt{\alpha^{2}-\lambda}}\bigg]
\end{split}
\end{equation*}
The inverse transforms above given functions are directly related to the Mittag-Leffler functions \cite{RC}. We have the first term of the convolution is the following
\begin{equation*}
    \frac{1}{2\sqrt{\alpha^{2}-\lambda}}E_{\rho,\rho}\bigg((-\alpha+\sqrt{\alpha^{2}-\lambda})t^{\rho}\bigg) -\frac{1}{2\sqrt{\alpha^{2}-\lambda}}E_{\rho,\rho}\bigg((-\alpha-\sqrt{\alpha^{2}-\lambda})t^{\rho}\bigg).
\end{equation*}

Plugging this function into \eqref{y1} and combining it with \eqref{y0} we have
$$
y(t)=\frac{(\sqrt{\alpha^{2}-\lambda}+\alpha)\varphi_{1}}{2\sqrt{\alpha^{2}-\lambda}}E_{\rho,1}((-\alpha+\sqrt{\alpha^{2}-\lambda})t^{\rho})+\frac{(\sqrt{\alpha^{2}-\lambda}-\alpha)\varphi_{1}}{2\sqrt{\alpha^{2}-\lambda}}E_{\rho,1}((-\alpha-\sqrt{\alpha^{2}-\lambda})t^{\rho})
$$
$$
+\frac{1}{2\sqrt{\alpha^{2}-\lambda}}\big[E_{\rho,1}((-\alpha+\sqrt{\alpha^{2}-\lambda})t^{\rho})-E_{\rho,1}((-\alpha-\sqrt{\alpha^{2}-\lambda})t^{\rho})  \big]\varphi_{0}
$$
$$
+\frac{1}{2\sqrt{\alpha^{2}-\lambda}}\int_{0}^{t}(t-\tau)^{\rho-1}E_{\rho,\rho}((-\alpha+\sqrt{\alpha^{2}-\lambda})(t-\tau)^{\rho})g(\tau)d\tau
$$
$$
-\frac{1}{2\sqrt{\alpha^{2}-\lambda}}\int_{0}^{t}(t-\tau)^{\rho-1}E_{\rho,\rho}((-\alpha-\sqrt{\alpha^{2}-\lambda})(t-\tau)^{\rho})g(\tau)d\tau.
$$

\textbf{Case 2}. Let  $\alpha^{2}=\lambda$. In this case (\ref{laplas}) has the following form
\[
\hat{y}(s)=\frac{\hat{g}(s)+s^{2\rho-1}{y}(0)+s^{\rho-1}\lim_{t\to 0}D^{\rho}_{t}y(t)+2\alpha s^{\rho-1}y(0)}{(s^{\rho}+\alpha)^{2}}.
\]
Therefore
\[
\hat{y}(s)=\frac{s^{\rho-1}}{s^{\rho}+\alpha}y(0)+\frac{\alpha s^{\rho-1}}{(s^{\rho}+\alpha)^{2}}y(0)+\frac{s^{\rho-1}}{(s^{\rho}+\alpha)^{2}}\lim_{t \to 0}D^{\rho}_{t}y(t)+\frac{1}{(s^{\rho}+\alpha)^{2}}\hat{g(s)}.
\]
Passing to the inverse Laplace transform (see \cite{Mainardi},p.226,E67):

\[
y(t)=L^{-1}\big[\frac{s^{\rho-1}}{s^{\rho}+\alpha}y(0)\big]+L^{-1}\big[\frac{\alpha s^{\rho-1}}{(s^{\rho}+\alpha)^{2}}y(0)\big]+L^{-1}\big[\frac{s^{\rho-1}}{(s^{\rho}+\alpha)^{2}}\lim_{t \to 0}D^{\rho}_{t}y(t)\big]+L^{-1}\big[\frac{1}{(s^{\rho}+\alpha)^{2}}\hat{g(s)}\big],
\]
one has
$$
y(t)=E_{\rho,1}(-\alpha t^{\rho})\varphi_{1}+\alpha t^{\rho} E_{\rho,1+\rho}^{2}(-\alpha t^{\rho})\varphi_{1}+t^{\rho}E_{\rho,1+\rho}^{2}(-\alpha t^{\rho})\varphi_{0}+
\int_{0}^{t}(t-\tau)^{2\rho-1}E^{2}_{\rho,2\rho}(-\alpha(t-\tau)^{\rho})g(\tau)d\tau.
$$
\end{proof}
\begin{lem} Let $g(t)\in C[0,T]$.Then the unique solution of the Cauchy problem
\begin{equation}\label{Integro}
\begin{cases}
  & D_{t}^{\rho }u(t)+2\alpha u(t)+\alpha^{2}J^{-\rho}_{t}u(t)=J^{-\rho}_{t}g(t), \quad  0<t\le T; \\
 & u(0)=0, \\
\end{cases}
\end{equation}
with $0<\rho<1$ and $\alpha\in \mathbb{C}$ has the form
\[
u(t)=\int_{0}^{t}(t-\tau)^{2\rho-1}E^{2}_{\rho,2\rho}(-\alpha(t-\tau)^{\rho})g(\tau)d\tau.
\]
\end{lem}
\begin{proof} Let us apply the Laplace transform to equation \eqref{Integro}. Then equation (\ref{Integro}) becomes:
\[
s^{\rho}\hat{u}(s)-s^{\rho-1}u(0)+2\alpha\hat{u}(s)+\alpha^{2}s^{-\rho}\hat{u}(s)=s^{-\rho}\hat{g}(s),
\]
it follows from this
\[
\hat{u}(s)=\frac{s^{-\rho}\hat{g}(s)}{s^{\rho}+2\alpha+\alpha^{2}s^{-\rho}}=\frac{\hat{g}(s)}{(s^{\rho}+\alpha)^{2}}.
\]
Passing to the inverse Laplace transform we obtain:
\[
u(t)=L^{-1}\big[\frac{1}{(s^{\rho}+\alpha)^{2}} \big]\ast L^{-1}[\hat{g}(s)].
\]
First term in the convolution is known (see \cite{Mainardi},p.226,E67) and one has
\[
u(t)=\int_{0}^{t}(t-\tau)^{2\rho-1}E^{2}_{\rho,2\rho}(-\alpha(t-\tau)^{\rho})g(\tau)d\tau.\qedhere
\]
\end{proof}
\begin{lem} The solution to the Cauchy problem
\begin{equation}\label{Dervitavi1}
\begin{cases}
  & D_{t}^{\rho }u(t)-\lambda u(t)=f(t),\quad  0<t\le T; \\
 & u(0)=0, \\
\end{cases}
\end{equation}
with $0<\rho<1$ and $\lambda\in\mathbb{C} $ has the form
\[
u(t)=\int_{0}^{t}(t-\tau)^{\rho-1}E_{\rho,\rho}(\lambda(t-\tau)^{\rho})f(\tau)d\tau.
\]
\end{lem}
The proof of this lemme for $\lambda\in\mathbb{R}$ can be found in (see \cite{Kil},p.231). In complex case similar ideas will lead us to the same conclusion.

Consider the operator $E_{\rho, \mu} (t^{\rho} A): H\rightarrow H$ defined by the spectral theorem of J. von Neumann:
\[
E_{\rho, \mu} (t^{\rho} A)g = \sum\limits_{k=1}^\infty E_{\rho,\mu} (t^{\rho}\lambda_{k}) g_k v_k,
\]
here and everywhere below, by $g_k$  we will denote the Fourier coefficients of a vector $g\in H$: $g_k=(g,v_k)$.
\begin{lem}\label{Estimate1} Let $\alpha>0$. Then for any $g(t)\in C(H)$  one has $E_{\rho,\mu}(-S t^{\rho})g(t)\in C(H)$
and  $SE_{\rho,\mu}(-St^{\rho})g(t)\in C((0,T];H)$.  Moreover,the following estimates hold:
\begin{equation}\label{ES}
||E_{\rho, \mu} (-t^{\rho}S)g(t)||_{C(H)}\leq M||g(t)||_{C(H)},
\end{equation}
\begin{equation}\label{SES}
 ||SE_{\rho, \mu} (-t^{\rho} S)g(t)||\leq C_{1}t^{-\rho}||g(t)||_{C(H)}, \quad t>0.
\end{equation}
If $g(t)\in D(A^{\frac{1}{2}})$ for all $t\in[0,T]$, then
\begin{equation}\label{SES1}
||SE_{\rho, \mu} (-t^{\rho} S)g(t)||_{C(H)} \leq C_{2}\max_{0\leq t \leq T}||g(t)||_{\frac{1}{2}},
\end{equation}
\begin{equation}\label{AES}
||AE_{\rho, \mu} (-t^{\rho} S)g(t)|| \leq C_{3}t^{-\rho}\max_{0\leq t \leq T}||g(t)||_{\frac{1}{2}}, \quad t>0.
\end{equation}
Here $S$ has two states: $S^{-}$ and $S^{+}$,
\[
S^{-}=\alpha I-(\alpha^{2}I-A)^{\frac{1}{2}},\quad  S^{+}=\alpha I+(\alpha^{2}I-A)^{\frac{1}{2}}.
\]
\end{lem}
\begin{proof}By Parseval's equality one has
\[
||E_{\rho, \mu} (-S^{-} t^{\rho})g(t)||^2=||\sum\limits_{k=1}^\infty E_{\rho, \mu} (-(\alpha-\sqrt{\alpha^{2}-\lambda_{k}})t^{\rho}) g_k(t) v_k||^2=\sum\limits_{k=1}^\infty |E_{\rho, \mu} (-(\alpha-\sqrt{\alpha^{2}-\lambda_{k}})t^{\rho}) g_k(t)|^2.
\]
According to Corollary \ref{cor}, we write the following
\[
||E_{\rho, \mu} (-(\alpha-\sqrt{\alpha^{2}-\lambda_{k}})t^{\rho})g(t)||^2\leq M^{2} \sum\limits_{k=1}^\infty \bigg|\frac{ g_k(t)}{1+|\alpha-\sqrt{\alpha^{2}-\lambda_{k}}|t^{\rho}}\bigg|^2\leq M^{2} ||g(t)||^2,
 \]
 which concludes the assertion \eqref{ES}.
On the other hand,
\[
||S^{-}E_{\rho, \mu} (-t^{\rho} S^{-})g(t)||^{2}\leq M^{2}\sum\limits_{k=1}^\infty \frac{ |\alpha-\sqrt{\alpha^{2}-\lambda_{k}}|^{2}|g_k(t)|^{2}}{(1+t^{\rho}|\alpha-\sqrt{\alpha^{2}-\lambda_{k}|})^{2}},
\]
\[
u_{\lambda_{k}}(t)=\frac{ |\alpha-\sqrt{\alpha^{2}-\lambda_{k}}|^{2}|g_k(t)|^{2}}{(1+t^{\rho}|\alpha-\sqrt{\alpha^{2}-\lambda_{k}|})^{2}} \quad  \underset{\lambda_{k}\to \infty}{\sim} \quad v_{\lambda_{k}}(t)=\frac{ \lambda_{k}|g_k(t)|^{2}}{(1+t^{\rho}\sqrt{\lambda_{k}})^{2}},
\]
\[
\sum_{k=1}^{\infty}v_{\lambda_{k}}(t)=\sum_{k=1}^{\infty}\frac{ \lambda_{k}|g_k(t)|^{2}}{(1+t^{\rho}\sqrt{\lambda_{k}})^{2}}\leq t^{-2\rho}||g(t)||^{2}, \quad t>0.
\]

We say the seq ${u_{\lambda_{k}}}$ is eqvuivalent to the seq ${v_{\lambda_{k}}}$ if
\[
\lim_{\lambda_{k}\to \infty}\frac{{u_{\lambda_{k}}}}{{v_{\lambda_{k}}}}=1, \quad   {u_{\lambda_{k}}} \sim {v_{\lambda_{k}}}.
\]
Therefore
\[
||S^{-}E_{\rho, \mu} (-t^{\rho} S^{-})g(t)||^{2}\leq M^{2}C t^{-2\rho}||g(t)||^{2}=C_{1}t^{-2\rho}||g(t)||^{2}, \quad t>0.
\]
Obviously, if $g(t)\in D(A^{\frac{1}{2}})$ for all $t\in [0,T]$, then
\[
||S^{-}E_{\rho, \mu} (-t^{\rho} S^{-})g(t)||_{C(H)} \leq C_{2}\max_{0\leq t \leq T}||g(t)||_{\frac{1}{2}},
\]
\[
||AE_{\rho, \mu} (-t^{\rho} S^{-})g(t)|| \leq C_{3}t^{-\rho}\max_{0\leq t \leq T}||g(t)||_{\frac{1}{2}}, \quad t>0.
\]
A similar estimate is proved in exactly the same way with the operator $S^{-}$ replaced by the operator $S^{+}$.
\end{proof}

\begin{lem}\label{RES} Let   $\alpha>0$ and   $\lambda_{k}\neq \alpha^{2}$,  for all $k$. Then for  any $g(t)\in C(H)$  one has $R^{-1}E_{\rho,\mu}(-St^{\rho})g(t), SR^{-1}E_{\rho,\mu}(-St^{\rho})g(t)\in C(H)$ and $AR^{-1}E_{\rho, \mu} (-t^{\rho} S)g(t)\in C((0,T],H)$. Moreover,the following estimates hold:
\begin{equation}\label{RES}
||R^{-1}E_{\rho, \mu} (-t^{\rho} S)g(t)||_{C(H)} \leq C_{4} ||g(t)||_{C(H)},
\end{equation}
\begin{equation}\label{SRES}
||SR^{-1}E_{\rho, \mu} (-t^{\rho} S)g(t)||_{C(H)} \leq C_{5} ||g(t)||_{C(H)},
\end{equation}
\begin{equation}\label{ARES}
||AR^{-1}E_{\rho, \mu} (-t^{\rho} S)g(t)|| \leq C_{6}t^{-\rho}||g(t)||_{C(H)}, \quad t>0.
\end{equation}
Here
\[
 R^{-1}=(\alpha^{2}I-A)^{-\frac{1}{2}}.
\]
\end{lem}
\begin{proof}In proving the lemma, we use Parseval's equality and Corollary \ref{cor} similarly to the proof of  Lemma \ref{Estimate1}:
\[
||R^{-1}E_{\rho, \mu} (-t^{\rho} S^{-})g(t)||^{2}\leq M^{2} \sum\limits_{k=1}^\infty \bigg|\frac{1}{\sqrt{\alpha^{2}-\lambda_{k}}}\frac{g_k(t)}{1+t^{\rho}|\alpha-\sqrt{\alpha^{2}-\lambda_{k}}|}\bigg|^{2},
\]
\[
u_{\lambda_{k}}(t)=\frac{|g_k(t)|^{2}}{|\sqrt{\alpha^{2}-\lambda_{k}}|^{2}|(1+t^{\rho}|\alpha-\sqrt{\alpha^{2}-\lambda_{k}|})^{2}} \quad  \underset{\lambda_{k}\to \infty}{\sim} \quad v_{\lambda_{k}}(t)=\frac{|g_k(t)|^{2}}{\lambda_{k}(1+t^{\rho}\sqrt{\lambda_{k}})^{2}},
\]
\[
\sum_{k=1}^{\infty}v_{\lambda_{k}}(t)=\sum_{k=1}^{\infty}\frac{ |g_k(t)|^{2}}{\lambda_{k}(1+t^{\rho}\sqrt{\lambda_{k}})^{2}}\leq C^{*} ||g(t)||^{2}.
\]
Therefore
\[
||R^{-1}E_{\rho, \mu} (-t^{\rho} S^{-})g(t)||^{2}\leq M^{2}C^{*} ||g(t)||^{2}_{C(H)}=C_{4}||g(t)||^{2}_{C(H)}.
\]
Similarly,
\[
||S^{-}R^{-1}E_{\rho, \mu} (-t^{\rho} S^{-})g(t)||^{2}\leq M^{2}\sum\limits_{k=1}^\infty \bigg|\frac{|\alpha-\sqrt{\alpha^{2}-\lambda_{k}}|}{\sqrt{\alpha^{2}-\lambda_{k}}}\frac{g_k(t)}{1+t^{\rho}|\alpha-\sqrt{\alpha^{2}-\lambda_{k}}|}\bigg|^{2},
\]
\[
u_{\lambda_{k}}(t)=\frac{|\alpha-\sqrt{\alpha^{2}-\lambda_{k}}|^{2}|g_k(t)|^{2}}{|\sqrt{\alpha^{2}-\lambda_{k}}|^{2}|(1+t^{\rho}|\alpha-\sqrt{\alpha^{2}-\lambda_{k}|})^{2}} \quad  \underset{\lambda_{k}\to \infty}{\sim} \quad  v_{\lambda_{k}}(t)=\frac{|g_k(t)|^{2}}{(1+t^{\rho}\sqrt{\lambda_{k}})^{2}},
\]
\[
\sum_{k=1}^{\infty}v_{\lambda_{k}}(t)=\sum_{k=1}^{\infty}\frac{ |g_k(t)|^{2}}{(1+t^{\rho}\sqrt{\lambda_{k}})^{2}}\leq ||g(t)||^{2}.
\]

It remains to prove estimate (\ref{ARES}).  We consider the case with the operator $S^{-}$. We have
\[
||AR^{-1}E_{\rho, \mu} (-t^{\rho} S^{-})g(t)||^{2}\leq M^{2}\sum\limits_{k=1}^\infty \bigg|\frac{1}{\sqrt{\alpha^{2}-\lambda_{k}}}\frac{\lambda_{k}g_k(t)}{1+t^{\rho}|\alpha-\sqrt{\alpha^{2}-\lambda_{k}}|}\bigg|^{2},
\]
\[
u_{\lambda_{k}}(t)=\frac{\lambda_{k}^{2}|g_k(t)|^{2}}{|\sqrt{\alpha^{2}-\lambda_{k}}|^{2}|(1+t^{\rho}|\alpha-\sqrt{\alpha^{2}-\lambda_{k}|})^{2}} \quad  \underset{\lambda_{k}\to \infty}{\sim} \quad  v_{\lambda_{k}}(t)=\frac{\lambda_{k}^{2}|g_k(t)|^{2}}{\lambda_{k}(1+t^{\rho}\sqrt{\lambda_{k}})^{2}},
\]
\[
\sum_{k=1}^{\infty}v_{\lambda_{k}}(t)=\sum_{k=1}^{\infty}\frac{ \lambda_{k}|g_k(t)|^{2}}{(1+t^{\rho}\sqrt{\lambda_{k}})^{2}}\leq t^{-2\rho} ||g(t)||^{2}, \quad t>0.
\]
then
\[
||AR^{-1}E_{\rho, \mu} (-t^{\rho} S^{-})g(t)||^{2}\leq M^{2}C^{\ast\ast} t^{-2\rho} ||g(t)||^{2}_{C(H)}=C_{6}t^{-2\rho}||g(t)||^{2}_{C(H)}, \quad t>0.
\]
Similar estimates are proved in exactly the same way for the operator $S^{+}$.
\end{proof}

\begin{lem}\label{aep} Let $\alpha>0$ and $\lambda_{k}\neq \alpha^{2}$,  for all $k$. Then for  any $g(t)\in C([0,T]; D(A^\epsilon))$ for some $0< \epsilon < 1 $. Then
\begin{equation}\label{A}
\bigg|\bigg|\int\limits_0^t(t-\tau)^{\rho-1} AR^{-1}E_{\rho, \rho}(-(t-\tau)^\rho S)g(\tau) d\tau   \bigg|\bigg|\leq C \max\limits_{0\leq t \leq T} ||g(t)||_\epsilon.
\end{equation}
\begin{equation}\label{SR}
\bigg|\bigg|\int\limits_0^t(t-\tau)^{\rho-1} SR^{-1}E_{\rho, \rho}(-(t-\tau)^\rho S)g(\tau) d\tau   \bigg|\bigg|\leq C \max\limits_{0\leq t \leq T} ||g(t)||_{\epsilon}.
\end{equation}
\begin{equation}\label{R}
\bigg|\bigg|\int\limits_0^t(t-\tau)^{\rho-1}R^{-1}E_{\rho, \rho}(-(t-\tau)^\rho S)g(\tau) d\tau   \bigg|\bigg|\leq C \max\limits_{0\leq t \leq T} ||g(t)||_{\epsilon}.
\end{equation}
\end{lem}
\begin{proof}Let
    \[
        S_j(t)= \sum\limits_{k=1}^j
        \left[\int\limits_{0}^t\eta^{\rho-1}E_{\rho,\rho}(-(\alpha-\sqrt{\alpha^{2}-\lambda_{k}})
        \eta^{\rho})g_k(t-\eta)d\eta\right] \frac{\lambda_{k}}{\sqrt{\alpha^{2}-\lambda_{k}}}  v_k.
    \]
 We may write
    $$
    ||S_j(t)||^2=\sum\limits_{k=1}^j \bigg|\frac{\lambda_{k}}{\sqrt{\alpha^{2}-\lambda_{k}}}\bigg|^{2}
    \left|\int\limits_{0}^t\eta^{\rho-1}E_{\rho,\rho}(-(\alpha-\sqrt{\alpha^{2}-\lambda_{k}})
    \eta^{\rho})g_k(t-\eta)d\eta\right|^2 .
    $$
Apply estimate (\ref{Large}) to large enough $\exists j_{0}$, $\forall k>j_{0}$  to obtain
 \[
    ||S_j(t)||^2\leq C\sum\limits_{k=j_{0}}^j \bigg[ \int\limits_0^t
    \eta^{2\varepsilon\rho-1}\frac{\lambda_{k}}{\sqrt{\lambda_{k}}}\lambda_k^{\varepsilon-\frac{1}{2}}|g_k(t-\eta)|
    d\eta\bigg]^2.
    \]
Minkowski's inequality implies
    \[
    ||S_j(t)||^2\leq
    C\bigg[\int\limits_0^t\eta^{2\varepsilon\rho-1}\bigg(\sum\limits_{k=j_{0}}^j
    \lambda_k^{2\varepsilon}|g_k(t-\eta)|^2\bigg)^{\frac{1}{2}}
    d\eta\bigg]^2\leq C
    \max\limits_{0\leq t \leq T}||g(t)||^2_{\varepsilon}.
    \]
    Since
     \[
    \int\limits_0^t(t-\tau)^{\rho-1} AR^{-1}E_{\rho, \rho}(-(t-\tau)^\rho S^{-})g(\tau) d\tau =\lim_{j\to\infty}S_{j}(t),
    \]
     this implies the assertion of the (\ref{A}). (\ref{SR}) and (\ref{R}) are obtained in the same way as in the proof of (\ref{A}) combining the fact that $D(A^{\varepsilon}) \subset D(A^{\varepsilon-1/2})\subset D(A^{\varepsilon-1}) $.
 \end{proof}
\begin{lem}\label{Int} Let $\alpha>0$ and  $g(t)\in C(H)$. Then
  \begin{equation}\label{J}
  \bigg |\bigg|J^{-\rho}_{t}\bigg(\int_{0}^{t}(t-\tau)^{2\rho-1}E^{2}_{\rho,2\rho}(-\alpha(t-\tau)^{\rho})g(\tau)d\tau\bigg)\bigg|\bigg|_{C(H)}\leq \frac{M}{\Gamma(\rho)}\frac{T^{3\rho}}{2\rho^{3}}(2+\rho)||g(t)||_{C(H)}.
\end{equation}
\end{lem}
\begin{proof} For convenience, let's denote the argument of $J_t^{-\rho}$ by
\[
F(t)=\int_{0}^{t}(t-\tau)^{2\rho-1}E^{2}_{\rho,2\rho}(-\alpha(t-\tau)^{\rho})g(\tau)d\tau.
\]
According to \eqref{Prahma}
\[
F(t)=\frac{1}{\rho}\int_{0}^{t}(t-\tau)^{2\rho-1}E_{\rho,2\rho-1}(-\alpha(t-\tau)^{\rho})g(\tau)d\tau+\frac{1+\rho}{\rho}\int_{0}^{t}(t-\tau)^{2\rho-1}E_{\rho,2\rho}(-\alpha(t-\tau)^{\rho})g(\tau)d\tau,
\]
then
\[
\max_{0\leq t\leq T}||J^{-\rho}_{t}F(t)||=\max_{0\leq t\leq T}||\int_{0}^{t}F(\tau)(t-\tau)^{\rho-1}d\tau||\leq\max_{0\leq t\leq T}\int_{0}^{t}||F(\tau)|||t-\tau|^{\rho-1}d\tau
\]
\[
\leq ||F(t)||_{C(H)}\max_{0\leq t\leq T}\int_{0}^{t}|t-\tau|^{\rho-1}d\tau=||F(t)||_{C(H)}\max_{0\leq t\leq T}\frac{t^{\rho}}{\rho}\leq\frac{T^{\rho}}{\rho}||F(t)||_{C(H)},
\]
Thus we need to estimate $\lVert F(t)\rVert_{C(H)}$ and it can be done as follows
\[
  ||F(t)||_{C(H)}\leq\frac{1}{\rho}\big|\big|\int_{0}^{t}(t-\tau)^{2\rho-1}E_{\rho,2\rho-1}(-\alpha(t-\tau)^{\rho})g(\tau)d\tau\big|\big|_{C(H)}
\]
\[
+\frac{1+\rho}{\rho}\big|\big|\int_{0}^{t}(t-\tau)^{2\rho-1}E_{\rho,2\rho}(-\alpha(t-\tau)^{\rho})g(\tau)d\tau\big|\big|_{C(H)}
\]
\[
\leq\frac{1}{\rho}\big|\big|E_{\rho,2\rho-1}(-\alpha(t-\tau)^{\rho})g(t)\big|\big|_{C(H)}\max_{0\leq t \leq T}\int_{0}^{t}|(t-\tau)|^{2\rho-1}d\tau
\]
\[
+\frac{1+\rho}{\rho}\big|\big|E_{\rho,2\rho}(-\alpha(t-\tau)^{\rho})g(t)\big|\big|_{C(H)}\max_{0\leq t \leq T}\int_{0}^{t}|(t-\tau)|^{2\rho}d\tau.
\]
Using estimate (\ref{ES})
\begin{equation}\label{Ft}
||F(t)||_{C(H)}\leq \frac{MT^{2\rho}}{2\rho^{2}}(2+\rho)||g(t)||_{C(H)}.
\end{equation}
\end{proof}

\section {\textbf{Proof of the theorem on the forward problem}}

In this section, we prove Theorem \ref{main1}.

\begin{proof}
In accordance with the Fourier method, we will seek the solution of this problem in the form
\begin{equation}\label{fure}
u(t) =\sum\limits_{k=1}^\infty T_k(t) v_k,
\end{equation}
where $T_k(t)$ is a solution of the problem
   \begin{equation}\label{Cauchy1}
\begin{cases}
  & (D_{t}^{\rho })^{2}T_{k}(t)+2\alpha D_{t}^{\rho }T_{k}(t)+\lambda_{k} T_{k}(t)=f_{k}(t), \\
 & \underset{t\to 0}{\mathop{\lim }}\,D_{t}^{\rho }T_{k}(t)={{\varphi }_{0k}},\\
 & T_{k}(0)={{\varphi }_{1k}}, \\
\end{cases}
\end{equation}
Apply Lemma \ref{eq2} to get
\begin{equation}\label{2}T_{k}(t)=
\begin{cases}
  & y_{1k}(t),\quad \alpha^{2}\neq \lambda_{k}; \\
 & y_{2}(t),\quad \alpha^{2}=\lambda_{k}.
\end{cases}
\end{equation}
Therefore we have two cases:\\
Case I:  \quad $\alpha^2\neq \lambda_k$ for all $k\in\mathbb{N}$ \\
\begin{equation}
u(t)=\frac{1}{2}\bigg[E_{\rho,1}(-S^{-}t^{\rho})+E_{\rho,1}(-S^{+}t^{\rho})\bigg]\varphi_{1}+\frac{\alpha}{2}\bigg[R^{-1}E_{\rho,1}(-S^{-}t^{\rho})-R^{-1}E_{\rho,1}(-S^{+}t^{\rho})\bigg]\varphi_{1}
\end{equation}
$$
+\frac{1}{2}\bigg[R^{-1}E_{\rho,1}(-S^{-}t^{\rho})-R^{-1}E_{\rho,1}(-S^{+}t^{\rho})\bigg]\varphi_{0}
$$

$$
+\frac{1}{2}\int_{0}^{t}(t-\tau)^{\rho-1}\bigg[R^{-1}E_{\rho,\rho}(-S^{-}(t-\tau)^{\rho})-R^{-1}E_{\rho,\rho}(-S^{+}(t-\tau)^{\rho})\bigg]f(\tau)d\tau.
$$
Case II: \quad $\exists k_0\in \mathbb{N}$ such that $\alpha^2=\lambda_{k_0}$.

For simplicity we assume that there is only one $\lambda_{k_0}$ of this kind. Then the solution is
\begin{equation}\label{f}
u(t)=\frac{1}{2}\bigg[\tilde{E}_{\rho,1}(-S^{-}t^{\rho})+\tilde{E}_{\rho,1}(-S^{+}t^{\rho})\bigg]\varphi_{1}+\frac{\alpha}{2}\bigg[R^{-1}\tilde {E}_{\rho,1}(-S^{-}t^{\rho})-R^{-1}\tilde{E}_{\rho,1}(-S^{+}t^{\rho})\bigg]\varphi_{1}
\end{equation}
$$
+\frac{1}{2}\bigg[R^{-1}\tilde{E}_{\rho,1}(-S^{-}t^{\rho})-R^{-1}\tilde{E}_{\rho,1}(-S^{+}t^{\rho})\bigg]\varphi_{0}+E_{\rho,1}(-\alpha t^{\rho})\varphi_{1k_0}v_{k_0}+\alpha t^{\rho} E^{2}_{\rho,1+\rho}(-\alpha t^{\rho})\varphi_{1k_0}v_{k_0}
$$
$$
+t^{\rho}E^{2}_{\rho,1+\rho}(-\alpha t^{\rho})\varphi_{0k}v_{k_0}+\int_{0}^{t}(t-\tau)^{2\rho-1}E^{2}_{\rho,2\rho}(-\alpha(t-\tau)^{\rho})f_{k_0}(\tau)v_{k_0}d\tau
$$
$$
+\frac{1}{2}\int_{0}^{t}(t-\tau)^{\rho-1}\bigg[R^{-1}\tilde{E}_{\rho,\rho}(-S^{-}(t-\tau)^{\rho})-R^{-1}\tilde{E}_{\rho,\rho}(-S^{+}(t-\tau)^{\rho})\bigg]f(\tau)d\tau,
$$
where we denote by $$\tilde{E}_{\rho,\mu}(-St^\rho)g=\sum\limits_{k\neq k_0} E_{\rho,\mu} (-(\alpha\pm\sqrt{\alpha^{2}-\lambda_{k}})t^{\rho}) g_k v_k.$$

In the case where there are several indices $k\in\mathbb{N}$ such that $\alpha^2=\lambda_{k}$, we can repeat the same argument with a slight modification in finite number of terms.

We claim that $u(t)$ is a solution of the problem (\ref{prob1}) in both above cases according to the definition \ref{def1}.

Since most of the terms of $u(t)$ is the same in both cases, it is sufficient to study the second case. All the estimates we use for showing the second case can be adjusted to show the first case.

According to (\ref{Prahma}), we write (\ref{f}) as follows:
\begin{equation}\label{1f}
u(t)=\frac{1}{2}\bigg[\tilde{E}_{\rho,1}(-S^{-}t^{\rho})+\tilde{E}_{\rho,1}(-S^{+}t^{\rho})\bigg]\varphi_{1}+\frac{\alpha}{2}\bigg[R^{-1}\tilde {E}_{\rho,1}(-S^{-}t^{\rho})-R^{-1}\tilde{E}_{\rho,1}(-S^{+}t^{\rho})\bigg]\varphi_{1}
\end{equation}
$$
+\frac{1}{2}\bigg[R^{-1}\tilde{E}_{\rho,1}(-S^{-}t^{\rho})-R^{-1}\tilde{E}_{\rho,1}(-S^{+}t^{\rho})\bigg]\varphi_{0}+E_{\rho,1}(-\alpha t^{\rho})\varphi_{1k_0}v_{k_0}+\frac{\alpha t^{\rho}}{\rho} E_{\rho,\rho}(-\alpha t^{\rho})\varphi_{1k_0}v_{k_0}
$$
$$
+\frac{2\alpha t^{\rho}}{\rho}E_{\rho,1+\rho}(-\alpha t^{\rho})\varphi_{1k_0} v_{k_0}+\frac{t^{\rho}}{\rho} E_{\rho,\rho}(-\alpha t^{\rho})\varphi_{0k_0}v_{k_0}+\frac{2 t^{\rho}}{\rho}E_{\rho,1+\rho}(-\alpha t^{\rho})\varphi_{0k_0} v_{k_0}
$$
$$
+\frac{1}{\rho}\int_{0}^{t}(t-\tau)^{2\rho-1}E_{\rho,2\rho-1}(-\alpha(t-\tau)^{\rho})f_{k_0}(\tau)v_{k_0}d\tau+\frac{1+\rho}{\rho}\int_{0}^{t}(t-\tau)^{2\rho-1}E_{\rho,2\rho}(-\alpha(t-\tau)^{\rho})f_{k_0}(\tau)v_{k_0}d\tau
$$
$$
+\frac{1}{2}\int_{0}^{t}(t-\tau)^{\rho-1}\bigg[R^{-1}\tilde{E}_{\rho,\rho}(-S^{-}(t-\tau)^{\rho})-R^{-1}\tilde{E}_{\rho,\rho}(-S^{+}(t-\tau)^{\rho})\bigg]f(\tau)d\tau.
$$

 Estimate for $\lVert u(t)\rVert_{C(H)}$ using (\ref{ES}),(\ref{RES}), Corollary \ref{cor}, (\ref{R})
$$
||u(t)||_{C(H)}\leq (M+\alpha C_{4})||\varphi_{1}||+C_{4}||\varphi_{0}||+\left(M+\frac{3\alpha M T^{\rho}}{\rho}\right)|\varphi_{1k_{0}}|+\frac{3MT^{\rho}}{\rho}|\varphi_{0k_{0}}|
$$
$$
+\frac{MT^{2\rho}}{\rho^{2}}(2+\rho)\max_{0\leq t \leq T}|f_{k_{0}}(t)|+C\max_{0\leq t \leq T}||f(t)||_{\epsilon}.
$$
It should be shown that this series converges after applying operator $A$ and the derivatives $(D^{\rho})^{2}_{t}$, $D^{\rho }_{ t}$.

Let us estimate $Au(t)$. If $S_{j}(t)$ is a partial sum of (\ref{1f}), then
$$
AS_{j}(t)=\frac{1}{2}\sum\limits_{\underset{k\ne {{k}_{0}}}{\mathop{k=1}}\,}^{j}\bigg[E_{\rho,1}(-(\alpha-\sqrt{\alpha^{2}-\lambda_{k}})t^{\rho})\varphi_{1k}+E_{\rho,1}(-(\alpha+\sqrt{\alpha^{2}-\lambda_{k}})t^{\rho})\varphi_{1k}
$$
$$
+\frac{\alpha}{\sqrt{\alpha^{2}-\lambda_{k}}}E_{\rho,1}(-(\alpha-\sqrt{\alpha^{2}-\lambda_{k}})t^{\rho})\varphi_{1k}-\frac{\alpha}{\sqrt{\alpha^{2}-\lambda_{k}}}E_{\rho,1}(-(\alpha+\sqrt{\alpha^{2}-\lambda_{k}})t^{\rho})\varphi_{1k}
$$
$$
+\frac{1}{\sqrt{\alpha^{2}-\lambda_{k}}}E_{\rho,1}(-(\alpha-\sqrt{\alpha^{2}-\lambda_{k}})t^{\rho})\varphi_{0k}-\frac{1}{\sqrt{\alpha^{2}-\lambda_{k}}}E_{\rho,1}(-(\alpha+\sqrt{\alpha^{2}-\lambda_{k}})t^{\rho})\varphi_{0k}
$$
$$
+\frac{1}{\sqrt{\alpha^{2}-\lambda_{k}}}\int_{0}^{t}(t-\tau)^{\rho-1}E_{\rho,\rho}(-(\alpha-\sqrt{\alpha^{2}-\lambda_{k}})(t-\tau)^{\rho})f_{k}(\tau)d\tau
$$
$$
-\frac{1}{\sqrt{\alpha^{2}-\lambda_{k}}}\int_{0}^{t}(t-\tau)^{\rho-1}E_{\rho,\rho}(-(\alpha+\sqrt{\alpha^{2}-\lambda_{k}})(t-\tau)^{\rho})f_{k}(\tau)d\tau\bigg]\lambda_{k}v_{k}
$$
$$
+E_{\rho,1}(-\alpha t^{\rho})\varphi_{1k_0}\lambda_{k_{0}}v_{k_0}+\frac{\alpha t^{\rho}}{\rho} E_{\rho,\rho}(-\alpha t^{\rho})\varphi_{1k_0}\lambda_{k_{0}}v_{k_0}
$$
$$
+\frac{2\alpha t^{\rho}}{\rho}E_{\rho,1+\rho}(-\alpha t^{\rho})\varphi_{1k_0} \lambda_{k_{0}}v_{k_0}+\frac{t^{\rho}}{\rho} E_{\rho,\rho}(-\alpha t^{\rho})\varphi_{0k_0}\lambda_{k_{0}}v_{k_0}+\frac{2 t^{\rho}}{\rho}E_{\rho,1+\rho}(-\alpha t^{\rho})\varphi_{0k_0}\lambda_{k_{0}} v_{k_0}
$$
$$
+\frac{1}{\rho}\int_{0}^{t}(t-\tau)^{2\rho-1}E_{\rho,2\rho-1}(-\alpha(t-\tau)^{\rho})f_{k_0}(\tau)\lambda_{k_{0}}v_{k_0}d\tau+\frac{1+\rho}{\rho}\int_{0}^{t}(t-\tau)^{2\rho-1}E_{\rho,2\rho}(-\alpha(t-\tau)^{\rho})f_{k_0}(\tau)\lambda_{k_{0}}v_{k_0}d\tau
$$

Using estimates (\ref{AES}), (\ref{ARES}), Corollary \ref{cor} and (\ref{A})  consequently for above given expression we get
\begin{equation*}
    \begin{split}
        ||AS_{j}(t)||\leq & C_3t^{-\rho}||\varphi_{1}||_{\frac{1}{2}}+\alpha  C_{6}t^{-\rho}||\varphi_{1}||+C_{6}t^{-\rho}||\varphi_{0}||+\alpha^{2}(M+\frac{3\alpha M T^{\rho}}{\rho})|\varphi_{1k_{0}}|\\
&+\frac{3\alpha^{2}MT^{\rho}}{\rho}|\varphi_{0k_{0}}|+\frac{\alpha^{2}MT^{2\rho}}{\rho^{2}}(2+\rho)\max_{0\leq t \leq T}|f_{k_{0}}(t)|+C\max_{0\leq t \leq T}||f(t)||_{\epsilon}, \quad t>0.
    \end{split}
\end{equation*}

Hence, it is sufficient to have $\varphi_{0}\in H$, $\varphi_{1}\in D(A^{\frac{1}{2}})$  and $f(t)\in C([0,T];D(A^{\epsilon}))$ for having $Au(t)\in C((0,T];H)$.

Let us now estimate $D^{\rho}_{t}u(t)$. If $S_{j}(t)$ is a partial sum of (\ref{1f}), then by (\ref{Dervitavi}), (\ref{Dervitavi1}) and (\ref{Integro}) we see that
$$
D^{\rho}_{t}S_{j}(t)=\frac{1}{2}\sum\limits_{\underset{k\ne {{k}_{0}}}{\mathop{k=1}}\,}^{j}\bigg[-(\alpha-\sqrt{\alpha^{2}-\lambda_{k}})E_{\rho,1}(-(\alpha-\sqrt{\alpha^{2}-\lambda_{k}})t^{\rho})\varphi_{1k}
$$
$$
-(\alpha+\sqrt{\alpha^{2}-\lambda_{k}})E_{\rho,1}(-(\alpha+\sqrt{\alpha^{2}-\lambda_{k}})t^{\rho})\varphi_{1k}
$$
$$
-\frac{\alpha(\alpha-\sqrt{\alpha^{2}-\lambda_{k}})}{\sqrt{\alpha^{2}-\lambda_{k}}}E_{\rho,1}(-(\alpha-\sqrt{\alpha^{2}-\lambda_{k}})t^{\rho})\varphi_{1k}+\frac{\alpha(\alpha+\sqrt{\alpha^{2}-\lambda_{k}})}{\sqrt{\alpha^{2}-\lambda_{k}}}E_{\rho,1}(-(\alpha+\sqrt{\alpha^{2}-\lambda_{k}})t^{\rho})\varphi_{1k}
$$
$$
-\frac{\alpha-\sqrt{\alpha^{2}-\lambda_{k}}}{\sqrt{\alpha^{2}-\lambda_{k}}}E_{\rho,1}(-(\alpha-\sqrt{\alpha^{2}-\lambda_{k}})t^{\rho})\varphi_{0k}+\frac{\alpha+\sqrt{\alpha^{2}-\lambda_{k}}}{\sqrt{\alpha^{2}-\lambda_{k}}}E_{\rho,1}(-(\alpha+\sqrt{\alpha^{2}-\lambda_{k}})t^{\rho})\varphi_{0k}
$$

$$
-\frac{\alpha-\sqrt{\alpha^{2}-\lambda_{k}}}{\sqrt{\alpha^{2}-\lambda_{k}}}\int_{0}^{t}(t-\tau)^{\rho-1}E_{\rho,\rho}(-(\alpha-\sqrt{\alpha^{2}-\lambda_{k}})(t-\tau)^{\rho})f_{k}(\tau)d\tau
$$
$$
+\frac{\alpha+\sqrt{\alpha^{2}-\lambda_{k}}}{\sqrt{\alpha^{2}-\lambda_{k}}}\int_{0}^{t}(t-\tau)^{\rho-1}E_{\rho,\rho}(-(\alpha+\sqrt{\alpha^{2}-\lambda_{k}})(t-\tau)^{\rho})f_{k}(\tau)d\tau\bigg]v_{k}
$$
$$
-\alpha E_{\rho,1}(-\alpha t^{\rho})\varphi_{1k_{0}}v_{k_{0}}-\frac{\alpha^{2}t^{\rho}}{\rho}E_{\rho,\rho}(-\alpha t^{\rho})\varphi_{1k_{0}}v_{k_{0}}-\frac{2\alpha^{2}t^{\rho}}{\rho}E_{\rho,\rho+1}(-\alpha t^{\rho})\varphi_{1k_{0}}v_{k_{0}}
$$
$$
-\frac{t^{\rho}\alpha}{\rho}E_{\rho,\rho}(-\alpha t^{\rho})\varphi_{0k_{0}}v_{k_{0}}-\frac{2\alpha t^{\rho}}{\rho}E_{\rho,\rho+1}(-\alpha t^{\rho})\varphi_{0k_{0}}v_{k_{0}}-2\alpha\int_{0}^{t}(t-\tau)^{2\rho-1}E^{2}_{\rho,2\rho}(-\alpha(t-\tau)^{\rho})f_{k_{0}}(\tau)v_{k_{0}}d\tau
$$
$$
-\alpha^{2}J^{-\rho}_{t}\bigg(\int_{0}^{t}(t-\tau)^{2\rho-1}E^{2}_{\rho,2\rho}(-\alpha(t-\tau)^{\rho})f_{k_{0}}(\tau)v_{k_{0}}d\tau\bigg)+J^{-\rho}_{t}f_{k_{0}}(t)v_{k_{0}}
$$

Applying the estimates (\ref{SES}), (\ref{SRES}), Corollary \ref{cor} and (\ref{SR}), (\ref{J}) for corresponding terms of above expression we have
$$
||D^{\rho}_{t}S_{j}(t)||\leq C_{5}||\varphi_{0}||+(C_{1}t^{-\rho}+\alpha C_{5})||\varphi_{1}||+\frac{3M\alpha^{2}T^{\rho}}{\rho}|\varphi_{1k_{0}}|+\frac{3M\alpha T^{\rho}}{\rho}|\varphi_{0k_{0}}|
$$
$$
+\frac{2M\alpha^{2} T^{2\rho}}{2\rho^{2}}(2+\rho)\max_{0\leq t\leq T}|f_{k_{0}}|+\frac{M T^{3\rho}(2+\rho)}{\Gamma(\rho)2\rho^{3}}\max_{0\leq t\leq T}|f_{k_{0}}|+\frac{TM^{\rho}}{\rho}\max_{0\leq t\leq T}|f_{k_{0}}|+C\max_{0\leq t \leq T}||f(t)||_{\epsilon}, \quad t>0.
$$

If $\varphi_{1},\varphi_{0}\in H$ and $f(t)\in C([0,T];D(A^{\epsilon}))$, then we have $D^{\rho}_{t}u(t)\in C((0,T];H))$.

Further, equation (\ref{prob1}) implies $(D^{\rho}_{t})^{2}u(t)=-2\alpha D^{\rho}_{t}u(t)-Au(t)+f(t)$.Therefore, arguing as above
we find that $(D^{\rho}_{t})^{2}u(t) \in C((0,T];H) $.

Using estimates (\ref{SES1}) and similar ideas as above estimate we have
$$
||D^{\rho}_{t}S_{j}(t)||\leq C_{5}||\varphi_{0}||+C_{2}||\varphi_{1}||_{\frac{1}{2}}+\alpha C_{5}||\varphi_{1}||+\frac{3M\alpha^{2}T^{\rho}}{\rho}|\varphi_{1k_{0}}|+\frac{3M\alpha T^{\rho}}{\rho}|\varphi_{0k_{0}}|
$$
$$
+\frac{2M\alpha^{2} T^{2\rho}}{2\rho^{2}}(2+\rho)\max_{0\leq t\leq T}|f_{k_{0}}|+\frac{M T^{3\rho}(2+\rho)}{\Gamma(\rho)2\rho^{3}}\max_{0\leq t\leq T}|f_{k_{0}}|+\frac{T^{\rho}}{\rho}\max_{0\leq t\leq T}|f_{k_{0}}|+C\max_{0\leq t \leq T}||f(t)||_{\epsilon}.
$$
If $\varphi_{1}\in D(A^{\frac{1}{2}})$,$\varphi_{0}\in H$ and $f(t)\in C([0,T];D(A^{\epsilon}))$, then we have $D^{\rho}_{t}u(t)\in C(H)$.

Let us prove the uniqueness of the solution. We use a standard technique based on the
completeness of the set of eigenfunctions $\{v_k\}$ in $H$.

Let $u(t)$ be a solution to the problem
\begin{equation}\label{Unique}
\begin{cases}
  & (D_{t}^{\rho })^{2}u(t)+2\alpha D_{t}^{\rho }u(t)+Au(t)=0, \quad  0<t\le T; \\
 & \underset{t\to 0}{\mathop{\lim }}\,D_{t}^{\rho }u(t)=0,\\
 & u(0)=0. \\
\end{cases}
\end{equation}
Set $u_{k}(t) = (u(t), v_{k})$. Then, by virtue of equation (\ref{Unique}) and the
selfadjointness of operator $A$,
\begin{equation}\label{Uniq1}\nonumber
(D^{\rho}_{t})^{2}u_{k}(t)=(\big(D^{\rho}_{t})^{2}u(t),v_{k}\big)=(-2\alpha D^{\rho}_{t}u(t)-Au(t),v_{k})=
\end{equation}
$$
=(-2\alpha D^{\rho}_{t}u(t),v_{k})-(Au(t),v_{k})=-2\alpha (D^{\rho}_{t}u(t),v_{k})-(u(t),Av_{k})=
$$
$$
=-2\alpha D^{\rho}_{t}u_{k}(t)-\lambda_{k}u_{k}(t).
$$
Hence, we have the following problem for $u_k(t)$:
\begin{equation}\label{Unique2}\nonumber
\begin{cases}
  & (D_{t}^{\rho })^{2}u_{k}(t)+2\alpha D_{t}^{\rho }u_{k}(t)+\lambda_{k}u(t)=0,\quad 0<t\le T; \\
 & \underset{t\to 0}{\mathop{\lim }}\,D_{t}^{\rho }u_{k}(t)=0,\\
 & u_{k}(0)=0. \\
\end{cases}
\end{equation}
Lemma \ref{eq2} implies that $u_k(t)\equiv 0$ for all $k$ . Consequently, due to
the completeness of the system of eigenfunctions $\{v_k\}$, we have $u(t)\equiv0$, as required.
\end{proof}

\section{Acknowledgement}
The authors are grateful to A. O. Ashyralyev for posing the
problem and they convey thanks to Sh. A. Alimov for discussions of
these results.
The authors acknowledge financial support from the  Ministry of Innovative Development of the Republic of Uzbekistan, Grant No F-FA-2021-424.

\end{document}